\documentclass[11pt]{amsart}
\usepackage[left=2cm,top=2cm,right=3cm,nofoot]{geometry}
\usepackage{amsmath,mathtools}%
\usepackage{amsfonts}%
\usepackage{amssymb}%
\usepackage{graphicx}
\usepackage{esint}
\usepackage{hyperref}
\usepackage{enumitem}
\usepackage{accents}
\usepackage{etoolbox}
\patchcmd{\subsection}{-.5em}{.5em}{}{}
\newtheorem{theorem}{Theorem}[section]
\theoremstyle{plain}

\newtheorem{corollary}[theorem]{Corollary}

\newtheorem{definition}[theorem]{Definition}

\newtheorem{lemma}[theorem]{Lemma}

\newtheorem{remark}[theorem]{Remark}

\numberwithin{equation}{section}
\theoremstyle{plain}

\usepackage{etoolbox}
\AtEndEnvironment{proof}{\setcounter{claim}{0}}
\newcommand{\re}{\mathbb{R}}

\begin{document}

\title[Allen-Cahn]{Interface foliation near a minimal isoparametric hypersurface for the Allen-Cahn equation}

\author[G. Henry]{Guillermo Henry}
\address{Departamento de Matem\'atica, FCEyN, Universidad de BuenosAires and IMAS,   CONICET-UBA, Ciudad Universitaria, Pab. I., C1428EHA, Buenos Aires, Argentina and CONICET, Argentina.}
\email{ghenry@dm.uba.ar}

\author[J. Petean]{Jimmy Petean}\thanks{J. Petean is supported by [SECIHTI proyecto CBF-2025-I-1449].}
\address{Centro de Investigaci\'{o}n en Matem\'{a}ticas, CIMAT, Calle Jalisco s/n, 36023 Guanajuato, Guanajuato, M\'{e}xico}
\email{jimmy@cimat.mx}

\begin{abstract} Let $(M,g)$ be a closed Riemannian manifold of positive Ricci curvature. Let $f$ be a proper isoparametric 
function on $M$, and $\Gamma$ the unique level set of $f$ which is a minimal hypersurface.
For any positive integer  $k$, and $\lambda  >0$ large enough, we construct a solution of the
Allen-Cahn equation $ \Delta u + \lambda (u-u^3 ) =0$ which is constant along the level sets of $f$ and has exactly
$k$ nodal components. We prove that the nodal components approach the minimal isoparametric hypersurface $\Gamma$ as
$\lambda \rightarrow \infty$ and the energy is uniformly bounded for all such solutions.
\end{abstract}
\maketitle
\section{Introduction}

We will study solutions of the Allen-Cahn equation

\begin{equation}\label{AC}
 \Delta_{g } u + \lambda (u -u^3   )= 0,
\end{equation}

\noindent
with $\lambda >0$, on a closed Riemannian manifold  $(M,g)$. Solutions are the critical points of the energy functional

\begin{equation}
\mathcal{E}_{\lambda} \  (u) = \frac{1}{\sqrt{\lambda}}  \int_M  \frac{ \| \nabla u \|^2}{2}  +  \sqrt{\lambda}  \int_M \frac{(1-u^2 )^2}{4}  .
\end{equation}

The functional and the equation  appeared, at first in Euclidean space,  in the study of modelling of phase transitions \cite{AllenCahn, AllenCahn2}. Note that 
there are three constant solutions, $u = 0, \pm 1$.  In the last years there has been great interest in the study of the Allen-Cahn equation in 
closed Riemannian manifolds and its relation with minimal hypersurfaces. See for instance the article by P. Gaspar and M. Guaraco \cite{GasparGuaraco}. 
The idea is that non-constant solutions with 
$\lambda $ large and uniformly bounded energy should be close to $\pm 1$ on regions separated by a minimal hypersurface, 
which is called
the {\it limit interface}. In  this direction L. Modica  proved in \cite{Modica}, that for a sequence of non-constant solutions of
minimal energy and $\lambda \rightarrow \infty$, 
the nodal set of the solutions converge to a minimal hypersurface $\Gamma$. The energy of the solutions converges to $\sigma . | \Gamma |$,
where $\sigma$ is a universal constant and $ | \Gamma |$ denotes the area of $\Gamma$.
More generally, J. E. Hutchinson and Y. Tonegawa considered in \cite{Hutchinson} a sequence of non-minimizing solutions, 
with $\lambda \rightarrow \infty$ such
that the energy is uniformly bounded, and show that the limit interface is a minimal hypersurface $\Gamma$, with some integer multiplicity $k$. 
This means that the energy of the solutions will converge to $k\sigma |\Gamma |$, and one can picture the nodal sets as $k$ connected components 
converging to $\Gamma$. Other results in this direction are proved for instance in \cite{Guaraco, Tonegawa}.

\bigskip

In the other direction, it is natural to consider the problem if for a given minimal hypersurface $\Gamma$ there exists a sequence of solutions, with
limit interface $\Gamma$. The fundamental result in this direction, in the setting of a closed Riemannian manifold, was
obtained by  F. Pacard and M. Ritor\'{e} in \cite{PacardRitore}: they prove that given a non-degenerate, separating,  
minimal hypersurface,
for $\lambda$ large enough, there exists a solution with nodal set close to $\Gamma$. Here nondegenerate means that the Jacobi
operator of $\Gamma$ has a trivial kernel. And separating means that the complement of $\Gamma$ is the union of two disjoint regions. As $\lambda \rightarrow \infty$ the solutions will converge to $1$ on one of the regions and to $-1$ on the other. 
Moreover, the energy will converge to $\sigma . |\Gamma |$, which means that the limit interface has multiplicity one. A description of this and related results
can be found in \cite{Pacard}.

 The previous result  was generalized by R. Caju and P. Gaspar in \cite{Caju} to the case when all solutions of the Jacobi equation on the minimal
hypersurface $\Gamma$ are generated by isometries: for all $\varepsilon$ small the authors proved the existence of a solution, with limit interface $\Gamma$ with 
multiplicity one. J. Chen and P. Gaspar in \cite{Chen} construct examples of degenerate minimal hypersurfaces which are not the limit interface of
a sequence of solutions to the Allen-Cahn equation.

There are important results also about the construction of solutions with higher multplicity. M. Del Pino, M. Kowalczyk, J. Wei and J. Yang proved in
\cite{DelPino} that given a nondegenerate, separating, minimal hypersurface $\Gamma$, for each positive integer k,  there is a sequence $\lambda_i \rightarrow \infty$ and
solutions $w_i$ of the Allen-Cahn equation with limit interface $\Gamma$ with multiplicity $k$ (in this case the solutions are constructed only for the sequence
$\lambda_i \rightarrow \infty$, not for every $\lambda$ large enough). In \cite{GuaracoMarquesNeves}, M. Guaraco, F. C. Marques and A. Neves show examples 
with multiplicity greater than one in symmetric cases (like the equator on the round sphere), where the solutions exist for all $\lambda$ large enough.

\bigskip

An important role is played by 1-dimensional solutions. On the real line the Allen-Cahn equation is
$u''+ \lambda (u-u^3 )=0$, which is actually equivalent for any $\lambda >0$ to the equation $u'' +u -u^3 =0$.
This last equation has a unique solution $u_0$, which is monotone on $\re$ with limit $\pm 1$, $u_0 (0)=0$. 
The dilated function $u_{\lambda} (t) = u(\sqrt{\lambda} t)$ solves th Allen-Cahn equation and we see that
as $\lambda\rightarrow \infty$ the solution $u_{\lambda}$ converges uniformly on campact sets of $\re - \{  0  \}$ to
$\pm 1$. One can then use this function for $\lambda$ large to construct an approximate solution to the Allen-Cahn 
in a small normal neighborhood of a minimal hypersurface $\Gamma$, which approaches $\pm 1$ away for $\Gamma$.

In this article we will construct solutions with limit interface an isoparametric minimal hypersurface, with any given 
multiplicity. Recall that a smooth function $f:M \rightarrow [t_0 , t_1 ]$ on a Riemannian manifold $(M,g)$ is called {\it isoparametric}, 
if it satisfies that
there exist smooth functions $a, b: [t_0 , t_1 ] \rightarrow \re$ such that  $\Delta f =a \circ f$ and $ \| \nabla f \|^2 = b\circ f$. 
Geometrically it means that the level sets of $f$ are equidistant and have
constant mean curvature. The only critical level sets of $f$ are the minimum and maximum, $f^{-1} (t_i )$, $i=0,1$, which
are called the {\it focal submanifolds} of $f$. More details about isoparametric functions will be given in Section 2.
We will assume that the mean curvature of $f^{-1} (t)$ is a monotone function of $t$. This
happens for instance for every isoparametric function on a closed Riemannian manifold of positive Ricci curvature,
as proved by J. Ge and Z. Tang in \cite{GeTang}. 
In this case there is a unique level set $\Gamma = f^{-1} (a_0 )$ which is a minimal hypersurface. We will prove
the existence of solutions, with any given multiplicity, with interface foliation $\Gamma$, which is called an {\it isoparametric minimal hypersurface}. 
The solutions we build will be $f$-invariant, which means that they are constant along the level sets of $f$. 
Restricted to $f$ invariant functions the Allen-Cahn equation on $M$ reduces to an ordinary differential equation with similarities to the 1-dimensional Allen-Cahn equation. 
We will construct the solutions studying the ordinary differential equation.

\bigskip

In Section 2 we will discuss isoparametric functions in general, and point out in particular that if we write an $f$-invariant 
function $u$ as  $u = \varphi \circ {\bf d}$, where ${\bf d}$ denotes  the distance to one of the focal submanifolds,
then $u$ solves  Equation \ref{AC} if and only if $\varphi$ solves

\begin{equation}\label{ODE1}
\varphi '' (r) + h(r) \varphi ' (r) + \lambda  (\varphi (r) -\varphi (r) ^3 ) = 0 \text{ on } [0,D],
\end{equation}

\noindent 
where $\lambda >0$. The function $h:(0,D)$ is the mean curvature of a corresponding level set of $f$, and it is a smooth function. 
We will assume that $h$ a strictly decreasing 
on the open interval $(0,D)$, $h'(t) <0$ for all $t\in (0,D)$,  verifying that 
$\lim_{t\rightarrow 0} t \ h(t)= n_1$ and $\lim_{t \rightarrow D} (D-t) \  h(t)  = -n_2$, for some positive integers $n_1$ and $n_2$. 
We call $a_0 \in (0,D)$ the unique 0 of $h$. We look for global solutions $\varphi$  to Equation \ref{ODE1}, which means that 
$\varphi$ extends to 0 and $D$, with $\varphi '(0) = \varphi ' (D) =0$.

Then in Section 3 we will use 
shooting techniques to prove, under the conditions stated above on the function $h$,  the existence on solutions of Equation \ref{ODE1} with any given number of zeroes:

\begin{theorem}\label{solutions}
Given a positive integer $k$ there is a positive constant $\lambda_0 = \lambda_0 (k)$ such that if $\lambda > \lambda_0$ then
there is a solution of Equation \ref{ODE1}  with exactly k zeroes. Therefore, if $f$ is an isoparametric function on a 
closed Riemannian manifold of positive Ricci curvature, then there is an $f$-invariant
solution $u$ of Equation \ref{AC} such that the nodal set of $u$ has exactly $k$ connected components, each of them being a level set of
$f$.
\end{theorem}

Next, in Section 4,  we will prove that given a fixed positive integer  $k$, for $\lambda $ large enough, all the zeroes of a solution  with exactly  
$k$ zeroes approach $a_0$:

\begin{theorem}\label{interface} Consider a positive integer $k$.
Fix $\varepsilon >0$. There exists $\lambda_0$ such that if $\varphi$ is a solution of Equation \ref{ODE2} with exactly $k$ zeroes and 
$\lambda \geq \lambda_0$ then if $\varphi (t_0 )=0$ then $t_0 \in (a_0 -\varepsilon , a_0 +\varepsilon )$.  Therefore all the connected components
of the corresponding $f$-invariant solution of Equation \ref{AC} converge to the minimal isoparametric hypersurface $\Gamma$.
\end{theorem}

Fos a solution $\varphi$ of Equation \ref{ODE1} we define its energy as

$$\mathcal{E}_{\lambda} (\varphi ) = \frac{1}{\sqrt{\lambda}} \int_0^D \frac{\varphi '(t) ^2}{2} dt + \sqrt{\lambda} \int_0^D \frac{(1-\varphi^2 (t) )^2}{4} dt .$$

We note that if $u= \varphi \circ {\bf d}$ is the corresponding $f$-invariant solution of Equation \ref{AC}, and $K$ is an upper bound for the areas of the
level sets of $f$ then $\mathcal{E}_{\lambda} (u) \leq K \mathcal{E}_{\lambda} (\varphi )$.
In Section 5 we will use the previous results to  prove that there is a uniform bound for the energy of solutions of Equation \ref{ODE1}
with a bounded number of zeroes, and form the previous comments we obtain:

\begin{theorem} Fix a positive integer $k$.
There is a constant $C$ such that if $u$ is an $f$-invariant  solution of Equation \ref{AC} such that the nodal set of $u$ has at most $k$ connected 
components,  then
$E_{\lambda} (u) \leq C$.
\end{theorem}

\section{Isoparametric functions}

In this section we will give a brief introduction to isoparametric functions on Riemanniana manifolds, and obtain the expression
of Equation \ref{AC} for invariant functions. More details can be found for instance in \cite{Jurgen, Henry}.

\medskip

Let ($M^n,g$) be a closed connected Riemannian manifold of dimension $n$.
A  smooth  function $f:M\rightarrow[t_0,t_1]$  is called {\it isoparametric} if  $|\nabla f|^2=b(f)$, $\Delta f=a(f)$ for 
some smooth functions $a,b :[t_0 , t_1 ] \rightarrow \re$. 
Isoparametric functions on general Riemannian manifolds were introduced and studied  by Q-M Wang in
\cite{Wang}, following the classical work by E. Cartan \cite{Cartan}, B. Segre \cite{Segre}, T. Levi-Civita \cite{Civita},
where isoparametric functions in space forms were studied. 

It was proved by Q-M Wang in
\cite{Wang}    that the only zeros of the function  $b:[t_0,t_1]\rightarrow \mathbb{R}_{\geq 0}$ are $t_0$ and $t_1$, which means that the only critical levels of the isoparametric function $f$
are its global minimum and maximum. Moreover $M_1 = f^{-1} (t_0 )$ and
$M_2 = f^{-1} (t_1 )$ are smooth submanifolds and are called the {\it focal submanifolds}
of $f$. We denote by $d_i$ the dimension of $M_i$. If $d_1, d_2 \leq n-2$ we call $f$ a {\it proper} isoparametric function, as in \cite{GeTang}.
We will assume from now on that $f$ is proper, which implies that all the level sets of $f$ are connected. The level sets are equidistant 
and have constant  mean curvature.
For instance if $g_0$ is the metric of constant curvature one on the unit sphere, then 
any of the coordinate functions is a proper isoparametric function. On the other hand the square of the coordinate function is
also an isoparametric function, but it is not proper.

Let $D=d_g(M_{1},M_{2})$ and  ${\bf d} : M \rightarrow [0, D ]$ denote the distance to $M_1$,
${\bf d} (x) =d_g(M_1 ,x)$. Note that it follows from the previous comments that the level sets of $ {\bf d}$ are the same as the level sets
of $f$. We will consider functions which are constant on the these level sets:

\begin{definition}
A function $u:M\rightarrow \mathbb{R}$ is called $f$-invariant if $u(x)=\phi({\bf d}(x))$ for some  function $\phi:[0,D ]\rightarrow \mathbb{R}$.
\end{definition}

An $f$-invariant function $u$ is of course determined by the corresponding function $\phi$.

The most familiar case of isoparametric functions comes from cohomogeneity
one isometric actions. Assume that $G$ acts isometrically on $(M,g)$ with regular orbits of 
codimension one and that the orbit space is  an interval. If $f$ is a smooth function which is $G$-invariant and
its only critical points are the two singular orbits, then $f$ is isoparametric, and a function is $f$-invariant
if and only if it is $G$-invariant.  Note that in this situation the singular orbits have codimension at least 2, and therefore
the isoparametric function $f$ is proper. 

\medskip

In the rest of the section we will give an outline of how to express Equation \ref{AC} for an $f$-invariant function, which amounts to 
give the expression for the Laplacian of  $u(x)=\phi({\bf d}(x))$ in terms of $\phi$. More details can be found for instance in
\cite{Jurgen}.

\bigskip

Denote by $\mathcal{B} = \{ \phi \in C^{2,\alpha} ([0,D ] ): \phi '(0) =0 = \phi '(D) \}$.

Let
$C^{2,\alpha}_f (M)$ be  the set of $C^{2,\alpha}$ functions on $M$ which are $f$-invariant.  Then the
application $\phi \mapsto u(x)=\phi({\bf d}(x))$ identifies $\mathcal{B}$ with $C^{2,\alpha}_f (M)$.

Now we 
express $\Delta u$ in terms of $\phi$. 

If  $x$ is not in the focal submanifolds, we have
\begin{align*}
\Delta(u(x))&=\Delta(\phi({\bf d}(x))\\
&=\phi^{''}  (  {\bf d}(x) ) \    (|\nabla {\bf d}(x))|^2) + \phi'  ({\bf d}(x) ) \ (\Delta ({\bf d}(x))).
\end{align*}

Also, as explained in \cite{Jurgen}, we have the expression 

$$\Delta {\bf d}(x) = \frac{-b'}{2b\sqrt{b}} (f(x)) |\nabla f|^2 (x)+\frac{\Delta f (x)}{\sqrt{b} (f(x) )} 
=\frac{1}{2\sqrt{b}}(-b'+2a) (f(x) )= h({\bf d} (x) ),$$

\noindent
where $h(t)$ is  the mean curvature  of the hypersurface ${\bf d}^{-1} (t)$. Therefore we have

$$\Delta(u(x)) =  ( \phi^{''} +h \phi' ) ({\bf d}(x)) .$$

The function $h$ is smooth in $(0,D)$. To study the case when $x$ is in the focal submanifolds 
we need to discuss  the asymptotic behaviour of the mean curvature close to focal submanifolds. 
Since each ${\bf d}^{-1} (t)$ is a tube over either $M_{1}$ or $M_{2}$ one can consider Fermi coordinates 
centered in $M_{1}$ or $M_{2}$. J. Ge and Z. Tang 
in \cite{GeTang2} compute the power series expansion formula for the shape operator of ${\bf d}^{-1} (t)$ with respect to the distance to $M_{1}$, and
use it to obtain the expression (see  \cite[Corollary 2.2]{GeTang2}) 

$$h(t)= \dfrac{codim(M_{1})-1}{t}+t(trace(A)+trace(B))+o(t^2),$$

\noindent
where $A,B$ are matrices independent of $t$ (if we consider the expansion close to $M_2$ instead of
$M_1$ we obtain a similar formula, changing the sign). Then we have the following asymptotic behaviour
of $h$ close to the focal varieties:

\begin{lemma}\label{mean-curvature}

$$\lim_{t\rightarrow 0}  t   \ h(t) = n-d_1 -1 \ \ \ ,  \ \ \ \ \lim_{t\rightarrow D} (t-D )   \ h(t) = n-d_2 -1. $$

\end{lemma}

\medskip

Therefore it follows that if $x\in M_1 $ then

\begin{equation}
\Delta u (x) =(n-d_1 )  \phi'' (0) 
\end{equation}

\noindent
and if $x\in M_2 $ then

\begin{equation}
\Delta u (x) =(n-d_2 )  \phi'' (D) 
\end{equation}

\bigskip

Finally  we have 

\begin{lemma}Let $u \in C^{2,\alpha}_f (M)$, $ u(x)=\phi({\bf d}(x))$, with $\phi \in \mathcal{B}$. Then $u$ is a solution of Equation
(\ref{AC}) if and only if the function $\phi$ satisfies 

\begin{equation}\label{ODE}
\phi ''+ h \phi ' + \lambda  (\phi - \phi^3  )=0.
\end{equation}

\noindent 
on $[0, D ]$, where for  $t\in (0,D)$, $h(t)$ is the mean curvature of  ${\bf d}^{-1} (t)$.

\end{lemma}

We will assume the condition that the function $h$ is stricly monotone. So in particular there is exactly
one point $a_0 \in (0,D)$ such that $h(a_0 )=0$. This level set $\Sigma = {\bf{d}}^{-1} (a_0 )$ is called the 
minimal isoparametric hypersurface corresponding to $f$.  Finally we point out that in case the Ricci curvature of
$(M,g)$ is positive it is always the case that $h$ is strictly monotone, as proved by J. Ge and Z. Tang  in \cite{GeTang}.

\section{Solutions of the  ordinary differential equation}

In this section we consider functions invariant by an isoparametric function so that  the equation can be reduced to an ordinary differential equation,
and construct solutions 
applying a shooting method to this ordinary differential equation. We consider then the equation

\begin{equation}\label{ODE2}
w'' (r) + h(r) w' (r) + \lambda  (w(r) -w(r) ^3 ) = 0 \text{ on } [0,D],
\end{equation}

\noindent 
where $\lambda >0$ and $h$ a strictly decreasing smooth function on the open interval $(0,D)$, $h'(t) <0$ for all $t\in (0,D)$,  verifying that 
$\lim_{t\rightarrow 0} t \ h(t)= n_1$ and $\lim_{t \rightarrow D} (D-t) \  h(t)  = -n_2$, for some positive integers $n_1$ and $n_2$. 
We call $a_0 \in (0,D)$ the unique 0 of $h$. We will be interested in global solutions, which verify that $w' (0) = w'(D) =0$.

\bigskip

For any value $d\in \re$ there is a uniquely defined solution $w_d$ such that 
$w_d (0)=d$ and $w_d '(0)=0$, defined in an open interval containing 0 (and we have a similar statement for the initial condition on $D$).

A global solution will verify that $w'(0) = w' (D) =0$ and it is uniquely determined
by the initial value $w(0)$.

\bigskip

For a solution $w$ of Equation \ref{ODE2} we consider the  function

$$\mathcal{D} (w)= \frac{(w')^2}{2} -\frac{\lambda}{4} (1- w^2 )^2 .$$

By a direct computation 

$$\mathcal{D} (w)' (r) = -h(r) (w'(r))^2 .$$

Therefore $ \mathcal{D} (w)$ is decreasing in $[0,a_0 ]$ and increasing in $[a_0 , D ]$. 
Note also that $\mathcal{D} (w) (0)= -(\lambda /4) (w(0)^2 -1)^2 \leq 0$.  And it is
strictly negative in case $| w(0) | \neq 1$. Note also that if $w(t)=\pm 1$ then $\mathcal{D} (w) (t) \geq 0$ and the equality holds if
and only if $w' (t) =0$ which would imply that $w$ is constant. It follows that for any solution $w$ with $w(0) \in (-1,1)$
we have that $w(t) \in (-1,1)$ for all $t$ in the maximal interval of definition of $w$ in $[0,a_0 ]$. It also follows from
$\mathcal{D} (w) \leq 0$ that $(w')^2 \leq (\lambda / 2)$.

The previous comments imply that for any $d\in [-1,1]$ the solution $w_d$ of the Equation \ref{ODE2} with initial conditions
$w(0)=d$, $w'(0)=0$ is defined at least in the interval $[0,a_0 ]$. Similarly, the solutions with initial  conditions on $D$ will be
defined at least in the interval $[a_0 , D ]$.

\medskip

{\bf Close to $d=0$, the linearized equation}.  Let $v = \frac{\partial w_d}{\partial d} :[0, D) \rightarrow \re$ be the solution of the linearized equation at $d=0$

\begin{equation}\label{ODEL}
v'' +  h(r) v' + \lambda  v= 0 \text{ on } [0, D),
\end{equation}

\noindent
with initial conditions $v'(0)=0$, $v(0)=1$. It is well known that there
is a sequence of  eigenvalues $0< \lambda_1 < \lambda _2, ...$ so that the corresponding solution $v_k$ when  $\lambda = \lambda_k$ is a global solution: $v_k$
extends to $D$ with $v_k ' (D) =0$. By the Sturm oscillation theorem for singular Sturm--Liouville problems, $v_k$ has exactly $k$ zeroes in $(0,D)$. 

\bigskip

The following result implies Theorem \ref{solutions}:

\begin{theorem}
If $\lambda \in  (\lambda_k , \lambda_{k+1} ]$ for each positive integer $j=1,\dots, k$ there is at  least one solution of
Equation \ref{ODE2}  which has exactly $j$ zeroes in $(0,D)$.

\end{theorem}

\begin{proof} To find solutions of Equation \ref{ODE2} for each $c, d \in [-1,1]$ we consider the solution $w_d$ which is defined on the 
interval $[0,a_0 ]$ and the solution $\widetilde{w}_c$ such that     $\widetilde{w}_c (D)=c$ and  $\widetilde{w}_c '(D)=0$,  
which is defined in the interval $[a_0 , D]$. If for some values $c , d \in [-1,1]$ we have that
$w_{d} (a_0 ) =\widetilde{w}_{c} (a_0 )$ and $w_{d} '  (a_0 ) =\widetilde{w}_{c} ' (a_0 )$ then $w=w_{d}$ can be extended as a global solution
of Equation \ref{ODE2}. 

To be more precise, we consider the initial value problem

\begin{equation}\label{Initial}
\left\{\begin{tabular}{cc}
$w_{i}''(r) + h(r)  w_i'(r) + \lambda  ( w_i -w_i^3) =0$& in $[0,a_0]$,\\
$w_i(0)=d, w_i'(0)=0,$ &
\end{tabular}\right.
\end{equation}
and the ``final'' value problem
\begin{equation}\label{Final}
\left\{\begin{tabular}{cc}
$w_f''(r) + h(r)  w_f'(r) + \lambda ( w_f -w_f^3 ) =0$& in $[a_0,D]$,\\
$w_f(D)=c, w_f'(D)=0,$ &
\end{tabular}\right.
\end{equation}
and we look for initial and final conditions $d$ and $c$ such that $w_i(a_0,d)=w_f(a_0,c)$ and $w'_i(a_0,d)=w'_f(a_0,c)$, so that, by uniqueness of the solution, we have a well defined 
global solution $w$ 
of Equation \ref{ODE2} given by $w(r)=w_i(r,d)$ if $r\in[0,a_0]$ and $w(r)=w_f(r,c)$ if $r\in[a_0, D]$.

We point out that the problem \eqref{Final} can be written as an initial condition problem having the form of \eqref{Initial}. Indeed, if we consider the function $\widetilde{h}(r)=-h(D-r)$, 
then $w_f$ solves \eqref{Final} if and only if $\omega(r)=w_f(D-r)$ solves the initial value problem
\begin{equation}\label{Eq:Singular backward equivalent}
\left\{\begin{tabular}{cc}
$\omega''(r) + \widetilde{h}(r) \omega'(r) + \lambda (\omega(r) -\omega^3 (r)) =0$& in $[0,D-a_0]$,\\
$\omega(0)=c, \omega'(0)=0,$ &
\end{tabular}\right.
\end{equation}
So understanding Equation \ref{Initial} is enough to also understand Equation \ref{Final}.

\bigskip

For $d\in[-1,1]$ let $w_d$ be the solution of \ref{ODE2}
with initial conditions $w_d (0)=d$, $w_d' (0)=0$. Note that $w_{-1}, w_0$ and $w_1$ are constant functions and
that $w_{-d}=-w_d$. Note also that if $d\neq -1 , 0 , 1$ and $r$ is a critical point of $w_d$ then $w_d (r) \neq -1 , 0 ,1$;
moreover $w_d$ has a local minimum (resp. maximum) at $r$ then  $w_d (r) \in (-1 , 0 ) $ (resp. $w_d (r) \in (0,1)$).

\medskip

Now consider the curve $I: [-1,1]  \rightarrow \re^2$ given by $I(d) = (w_d (a_0 ) , w_d ' (a_0 ) )$. Note that
$I(0) = (0,0)$, $I(1)=(1,0)$, $I(-d) =-I(d)$ and $I(d) \neq (0,0)$ if $d\neq 0$. It is then easy to see that we have
a well defined continuous function $\theta : (0, 1 ]  \rightarrow \re$ such that $\theta (1)=0$
and $\theta (d)$ gives an angle between $I(d)$ and the positive $x$-axis for any $d>0$. 

For any fixed $d\in (0,1]$ we can also define a continuous function $\alpha_d : [0, a_0 ] \rightarrow \re$
such that $\alpha_d (0)=0$ and for all $t \in [0,a_0 ]$ $\alpha_d (t)$ is an angle between $(w_d (t), w_d'(t))$ and
the positive $x$-axis. It is easy to check that $\theta (d) = \alpha_d (a_0 )$ (they both define the same angle, so the
difference is of the form $2k\pi$ for some integer $k$, they are continuous with respect to $d$ and coincide for $d=1$).  

Note that the function $\alpha_d$ can be extended beyond $a_0$, up to the maximal interval where
$w_d$ is defined. In case $w_d$ is a global solution $\alpha_d$ is defined in the whole interval
$[0,D ]$. In this last case $\alpha_d (D )$ is $-k\pi$ for some positive integer $k$ and the number
of zeroes of $w_d$ in $(0,D)$  is exactly $k$.

\medskip

Let $v:[0, \pi)$ be the solution of the linearized equation at $d=0$,
$v'' +  h(r) v' + \lambda  v= 0$  on $[0, D)$,
with initial conditions $v'(0)=0$, $v(0)=1$. Define a
function $\alpha_0  :[0, D) \rightarrow \re$ such that $\alpha (0) =0$ and $\alpha_0  (t) $ is an angle between
$(v(t),v'(t))$ and the positive $x$-axis. Note that $lim_{d\rightarrow 0} \  \theta (d) =
\lim_{d\rightarrow 0} \alpha_d (a_0 ) = \alpha_0  (a_0 )$.
It  follows that we can extend $\theta$ to a continuous function $\theta : [0,1] \rightarrow \re$.
Note also that if $\lambda =\lambda_k$ is
an eigenvalue of the linearized problem then $\alpha _0$ extends to $D$ and $\alpha_0 (D) = -k \pi$.

Note also that for $d<0$ we have that $\theta (-d) +\pi$ is an angle between $(w_d (a_0 ) ,w_d' (a_0 ))$ and
the positive $x$-axis. We can then extend the definition of $\theta$ to $[-1,1]$, but it is not continuous at 0:
$\lim_{d\rightarrow 0^-}  \theta (d) = \theta (0) +\pi$.

If $d\neq 0$ note that $w_d (a_0 )=0$ if and only if $\theta (d)  = \alpha_d (a_0 ) = -\frac{\pi}{2} - k \pi$ for some  integer $k$.
Actually, for each $t\in (0,a_0]$ we have that $w_d (t) =0$  if and only if $\alpha_d (t ) = -\frac{\pi}{2} - k \pi$ for some  integer $k$.

\bigskip

For $d\neq 0$ define $n(d)$ as the number of zeroes of $w_d$ in $[0, a_0 ]$. Note that $n(d)=n(-d)$ and that
$\theta (d)$ determines $n(d)$. This last statement follows from the identity $\theta (d)= \alpha_d (a_0 )$. For $d>0$, 
$\alpha_d (0) =0$ and  the solution $w_d$ has a local maximum in $0$. Then $w_d$ is strictly decreasing until it reaches a local minimum $t_2$,
where $w_d$ is negative. Up to his point  $\alpha_d$ is negative since $w_d'$ is negative. There is a unique point $t_1 \in (0,t_2 )$ such that $w_d (t_1 )=0$
and we have that  $\alpha_d (t_1 ) = -\pi /2$.  In the interval $(t_1 ,t_2 )$ we have that $\alpha_d$ is decreasing and $\alpha_d (t_2 ) =-\pi$.  
It should be clear from the previous comments that for $t\in (0,t_2 )$ we have that $\alpha_d (t)$ determines the number of zeores of $w_d$ in the 
interval $(0,t]$, and
we can repeat this argument until
reaching $a_0$.

\bigskip

Now we proceed  in the same way  for the solutions to problem \ref{Final} in $[a_0 ,D]$ with initial condition $w'(D)=0$. 

For $c\in [-1,1]$, denote by $\widetilde{w}_c$ the solution to the problem \ref{Final} and define the map $J(c):=(\widetilde{w}_c(a_0),\widetilde{w}_c'(a_0))$.

We have that  $J(1)=(1,0)$, $J(0)=(0,0)$, $J(c) \neq (0,0)$ if $c\neq 0$ and $J(-c)=-J(c)$. So, there is a well defined argument function $\vartheta :(0,1] \rightarrow \re$ such that
$\vartheta (1) =0$ and 

\[
J(c)=(\vert J(c)\vert\cos(\vartheta(c)),\vert J(c)\vert\sin(\vartheta(c)).
\]

\bigskip

For $c\in (0,1]$, we also have the function $\widetilde{\alpha}_c : [a_0 , D] \rightarrow \re$ such that $\widetilde{\alpha}_c (D)=0$ and $\widetilde{\alpha}_c (t)$ gives an
angle between $(\widetilde{w}_c (t), \widetilde{w}_c ' (t))$ and the positive $x$-axis.
Then $\widetilde{\alpha}_c (a_0 )= \vartheta (c)$.

We let $\widetilde{v}$ be the solution of the linearized equation with $\widetilde{v} (D )=1$, $\widetilde{v} ' (D)=0$.
Let $\widetilde{\alpha_0} :(0,D] \rightarrow \re$ be the continuous function such that 
$\widetilde{\alpha} (D) =0$  and $\widetilde{\alpha} (t)$ gives an angle between 
$(\widetilde{v} (t ), \widetilde{v}' (t ))$ and the positive $x$-axis.

All statement made for the initial value problem can be rewritten appropriately for the final value problem, including
the facts about the linearized equation.

Note that $\lim_{c\rightarrow 0} \vartheta (c) = \widetilde{\alpha_0} (a_0 )$

Let us denote $v=v^{\lambda}$, $\alpha_0 =\alpha_0^{\lambda}$, $\widetilde{v}=\widetilde{v}^{\lambda}$ and $\widetilde{\alpha}_0 =\widetilde{\alpha}_0^{\lambda}$, to
make explicit the dependence on $\lambda$. Note that with this notation we have that  $\alpha^{\lambda_k}_0 (a_0 ) - \widetilde{\alpha}^{\lambda_k}_0 (a_0 ) = -k\pi$. Indeed, $v^{\lambda_k}$ and $\widetilde{v}^{\lambda_k}$ correspond to eigenfunctions asociated with $\lambda_k$, then they are  proportional. There exist $s\in \mathbb{R}$ such that $\widetilde{v}^{\lambda_k}=v^{\lambda_k}$. Actually, $s=1/v^{\lambda_k}(D)$. Since $v^{\lambda_k}(0)=1$ and has exactly $k$ zeroes we have that the sign of $s$ is $(-1)^k$. Also, $(v^{\lambda_k})' (D)=0$, then $\alpha_0^{\lambda_k}(D)=-k\pi$. From this it follows that $\widetilde{\alpha}_0^{\lambda_k}(t)-\alpha_0^{\lambda_k}(t)+k\pi$ for any $t\in [0,D]$.

Since $\alpha_0^{\lambda}(a_0)-\widetilde{\alpha}_0^{\lambda}(a_0)$ is  decreasing as a function of $\lambda$, we have that $-(k+1)\pi<\alpha_0^{\lambda}(a_0)-\widetilde{\alpha}_0^{\lambda}(a_0)<-k\pi$ whenever  $\lambda\in (\lambda_k,\lambda_{k+1})$.

\vspace{.4cm}

Recall that if $w$ is a solution of Equation \ref{ODE2} then $-w$ is also a solution. We will consider solutions such that $w(D ) \in (0,1]$. It 
could be that $w(0) \in (0,1]$ or that $w(0) \in (-1,0)$. To prove the theorem we must show that if $\lambda \in (\lambda_k , \lambda_{k+1}]$ 
then for each $j=1 \dots , k$ there is at least one such solution with exactly $j$ zeroes in $(0,D)$. 

We define three  curves  $R_+ : (0,1] \rightarrow \re\times\re_{>0}$ ,
$R_- : (-1,0) \rightarrow \re\times\re_{>0}$, and $S:(0,1]\rightarrow\re\times\re_{>0}$, given by
\[
R_+ (d):=(\theta(d),\vert I(d) \vert)\qquad\text{,}\qquad S(c):=(\vartheta(c),\vert J(c)\vert)\qquad\text{and}\qquad R_- (d) = (\theta (-d) + \pi , \vert I(d) \vert ) .
\]

\noindent

The image of $R_-$ is obviously just  the horizontal traslation of the image 
of $R_+$ by $\pi$. Note that $R_+ (1)=(0,1)$, for any $t \in (0,1)$ we have that the first coordinate of $R_+ (t)$ is negative, and
$\lim_{t\rightarrow 0} R_+ (t) = (\alpha_0 (a_0 ),0)$. Similarly, $S(1)=(0,1)$ and 
the first coordinate of $S$ is non-negative. The curve $S$ starts at $(0,1)$  and ends at the point
$(\widetilde{\alpha_0} (a_0 ) , 0 )$. Observe that $R_+$, $R_-$ and $S$ give, respectively,  the polar coordinates of $( w_{d} (a_0 ), w_{d}' (a_0 ) )$,  $d\in (0,1]$, of $( w_{d} (a_0 ), w_{d}' (a_0 ) )$, $d\in [-1,0)$, and of $(\widetilde{w}_c  (a_0),\widetilde{w}_c '  (a_0) )$, $c\in (0,1]$.

Note also that the images of $R_+$ and $R_-$ do not intersect: if $R_+ (d_1 ) = R_- (d_2 )$ then we would have that
$( w_{d_1} (a_0 ), w_{d_1} ' (a_0 ) )= ( w_{d_2} (a_0 ), w_{d_2} ' (a_0 ) )$, which would imply by uniqueness of solutions that 
$w_{d_1} = w_{d_2 }$. Similarly, the image of $R_+$ cannot intersect its horizontal traslation by $j\pi$ for any integer $j$.

The length of the segment $[\alpha (a_0 ), \widetilde{\alpha} (a_0 ) ]$ falls in the interval $(k\pi , (k+1) \pi ]$. 
The segment together with $R_+$ and $S$ form a closed simple curve.

If $\vartheta (c) = j\pi$ for some integer $j$ then 
$\widetilde{w}_c '  (a_0) =0$, and therefore $| J(c) | < 1$. This implies that for  any $j \in \{ 1,\dots ,k \}$ the point $(j\pi , 1)$ is 
in the unbounded component of the complement of the simple closed curve. On the other hand, $\alpha_0 (a_0) +j\pi < 
\widetilde{\alpha_0} (a_0 )$. This shows that for each $j$, the horizontal  traslation of $R_+$ by $j\pi$ has points both in the bounded and the
unbounded components of the simple closed curve. Therefore, the horizontal traslation must intersect $S$.  If $d$ and  $c$ are such
that $(\theta(d) +j \pi ,\vert I(d) \vert) = (\vartheta(c),\vert J(c)\vert)$ then  $I(d)=J(c)$ when  $j$ is even, while 
$I(-d)=-I(d)= J(c)$ when $j$ is odd. Hence, depending on the parity of $j$, either $w_d$ or $w_{-d}$  extends to a global solution.  Assume first tha $w_d$, with $d\in (0,1)$, is the global solution, then $\alpha_d(a_0)-\widetilde{\alpha}_c(a_0)=-j\pi$. It follows that $\alpha_d(t)-\widetilde{\alpha}_c(t)=-j\pi$ for any $t$. In particular, $\alpha_d(D)=-j\pi$. Therefore, $w_d$ has exactly $j$ zeroes in $(0,D)$. If $w_{-d}$ is the global solution, then $-w_{-d}$ is also a global solution and has the same zeroes.

\end{proof}

\section{Nodal pieces}

In this section we consider the ordinary differential equation

\begin{equation}\label{ODE3}
w'' (t) + H(t) w' + \lambda  (w-w^3 ) (t) = 0 \text{ on } (0, T],
\end{equation}

\noindent
where we will  assume that $\lambda >0$, $\lim_{t\rightarrow 0} t H(t) = n_1$, for some positive integer $n_1$, $H:(0,T] \rightarrow [0,\infty )$ is smooth and decreasing, $H'(t)<0$ 
for all $t\in(0,T]$, and $H(T)=0$. We consider solutions 
on $(0,T]$ which extend continuously to 0, so we have the initial condition $w'(0)=0$. We are essentially considering the equation \ref{ODE2} of the previous section restricted to
the interval $[0,a_0]$.

\medskip

The purpose of this section is to prove that if $w$ is a solution of the equation with a bounded number of zeroes and $\lambda$ is  large, all the zeroes of the solution $w$
are close to the end point $T$:

\begin{theorem}\label{ODE-interface}
Fix $n_0 \in N$. For any $\varepsilon >0$, there exists $\lambda_0 >0$ such that: if $\lambda > \lambda_0$ and $w$ is a solution of 
Equation \ref{ODE3} with at most $n_0$ zeroes in $[0, T]$, then if $w(t_0 )=0$  we have that $t_0 \in (T-\varepsilon , T]$. 
\end{theorem}

\bigskip

The theorem is applied to solutions of Equation \ref{ODE2} in the interval $[0,a_0 ]$ and in the interval $[a_0 , D]$ to obtain
the following corollary, which proves Theorem \ref{interface}:

\begin{corollary} Fix $n_0 \in N$. For any $\varepsilon >0$, there exists $\lambda_0 >0$ such that: if $\lambda > \lambda_0$ and $w$ is a solution of 
Equation \ref{ODE2} with at most $n_0$ zeroes in $[0, D]$, then if $w(t_0 )=0$  we have that $t_0 \in (a_0 -\varepsilon , a_0 + \varepsilon )$. 
\end{corollary}

\bigskip

For a solution $w$ of Equation \ref{ODE3}, we define the discrepancy function $\mathcal{D} (w) :[0,T] \rightarrow \mathbb{R}$,
as  

$$\mathcal{D} (w) (t) =\frac{1}{2}  w'^2 (t) - \frac{\lambda}{4} (1- w^2 (t) )^2 .$$ 

\noindent
Then $\mathcal{D} (w)' (t) = -H(t)w'^2 (t) \leq 0$.

\medskip

We assume in the rest of this section that $w(0) \in [-1,1]$, $w' (0)=0$. Note that $\mathcal{D}(w)(0) =-\frac {\lambda}{4} (1-w(0)^2 )^2  \leq 0$ and so
$\mathcal{D} (w) (t) \leq 0$ for all $t\in [0,T]$. If $w(0) = 0, 1 $ or $-1$ then $w$ is constant. If  $w(0) \neq \pm 1$ then $\mathcal{D} (w)(0) <0$, which implies
that $\mathcal{D} (w)<0$ on $[0,T]$. Also note that if $w(t) =\pm 1$ then $\mathcal{D} (w)(t) \geq 0$. Therefore for all solutions $w$  with initial condition in
$(-1,1)$ we have that $w(t) \in (-1,1)$ for
any $t\in [0,T]$.

\medskip

We begin with the following elementary observation:

\begin{lemma}\label{elementary}
Let $w$ be a nonconstant solution of Equation \ref{ODE3}. Let $t_1 $ be a critical point of $w$ and $t_2 >t_1$. Then $w(t_2 )^2 <  w(t_1 )^2$.

\end{lemma}

\begin{proof}
Since $\mathcal{D}(w)$ is decreasing we have that

$$  - (\lambda /4) ((1- w^2 (t_2 ) )^2 \leq \mathcal{D} (w)(t_2 ) <  \mathcal{D} (w) (t_1 ) = - (\lambda /4) ((1- w^2 (t_1 ) )^2 . $$

It follows from this inequality that $w(t_2 )^2 <  w(t_1 )^2$.
\end{proof}

\medskip

Note also that if $w$ is a solution then $-w$ is also a solution, and $\mathcal{D}(w) = \mathcal{D}(-w)$. 
We will then often  assume that $w(0) \in [0,1]$. Note that in this case the previous lemma 
implies for instance that $w(0)$ is the global maximum of $w$. If $w(0)$ is very close to 1, then by the  continuity
of solutions of Equation \ref{ODE3} with respect to the initial conditions, we have that
$w$ will remain close to 1 near 0. The following result says that after $w$ decreases to reach a zero, the discrepancy 
function $\mathcal{D}(w)$ will be bounded above away from zero:

\begin{lemma}\label{discrepancybound} Consider an interval $I= [t_0, t_1] \subset [0, T )$. Assume that
$w$ is a solution of Equation \ref{ODE3} such that 
$ | w(t_1 ) -w(t_0 ) | > \frac{1}{2}$. Then $\mathcal{D} (w) (t_1 ) \leq \frac{- H(t_1  )}{ 4 (t_1 -t_0 ) }$. In particular if $\frac{1}{2} <w(0) <1$ and
$w(t_z) =0$ then $\mathcal{D}(w) (t_z ) \leq \frac{- H(t_z )}{ 4  \ t_z} $.
\end{lemma}

\begin{proof} We have that 

$$\mathcal{D} (w) (t_1 ) \leq \mathcal{D} (w) (t_1) - \mathcal{D} (w)(t_0 )=
\int_{t_0}^{t_1}\mathcal{D} (w)' (t) dt =-\int_{t_0}^{t_1} H(t) (w'(t))^2 dt .$$

Since $H$ is monotone decreasing we have that $H(t) \geq H(t_1  )$ for any $t\in I$ and therefore

$$\mathcal{D}(w) (t_1) \leq -H( t_1 ) \int_{t_0}^{t_1} w'(t)^2 dt .$$

By H\"{o}lder inequality 

$$ \left(  \int_{t_0}^{t_1} w'(t) dt  \right)^2   \leq  (t_1 -t_0 ) \int_{t_0}^{t_1} w'(t)^2 dt . $$

And we also have by hypothesis that

$$\frac{(w(t_1 ) -w(t_0 ))^2}{ t_1 -t_0 }  \geq \frac{1}{4(t_1 -t_0 )}.$$

Therefore  

$$\mathcal{D} (w) (t_1) \leq  \frac{-H( t_1 ) (w(t_1 ) -w(t_0 ))^2 }{ t_1 -t_0 } \leq  \frac{-H(t_1  )}{4(t_1 -t_0 )} .$$

\end{proof}

\medskip

Under the conditions of Theorem \ref{ODE-interface}, we want to prove that the length of the interval between the first zero of the solution $w$ of
Equation \ref{ODE3} and $T$ is close to zero. The next two lemmas prove that after the first zero of $w$, the length of the intervals where
$w$ is close to $\pm 1$ is close to 0.

\begin{lemma}\label{closeto1discrepancy}
Let $\delta >0$ small ($\delta <0.1$, for instance), $d>0$ be fixed. 
Let $w$ be a solution of Equation \ref{ODE3}. Assume there is an interval $I=[t_0 , t_1]$ of length 
$d=t_1 - t_0$ contained
in $(0 , T]$ where $w$ verifies $-1<w \leq -1 +\delta $. 

If  $w' \leq 0$ on $I$ then  $\mathcal{D} (w) (t_1 ) \geq \frac{-(w'(t_0 ))^2}{4\lambda (1-\delta )^2 d^2}$.

If $w' \geq 0$ on $I$ then $\mathcal{D} (w) (t_0) \geq \frac{-1}{4 \lambda (1-\delta )^2 d^2}   \left( w'(t_1 ) +H(t_0 )(w(t_1 )- w(t_0 )) \right)^2 .$ 

\medskip

In case $1-\delta <w <1$ a similar result is obtained by applying the previous result to $-w$.

\end{lemma}

\begin{proof} We have that $w'' = -Hw' -\lambda w (1-w^2 )$.

Assume first that  $w'  \leq 0$ on $I$. Then we have that $w'' \geq  \lambda (-w)(1-w^2 ) \geq  \lambda (1-\delta ) (1-w^2 ) 
\geq  \lambda (1-\delta ) (1-w(t_1)^2 ) $  
on the interval. Then 
it follows that  

$$ -w' (t_0 ) \geq w'(t_1 ) -w' (t_0)   = \int_{t_0}^{t_1} w''(t)dt  \geq  \lambda (1-\delta ) (1-w(t_1)^2 ) d . $$

Therefore

$$ 1-w(t_1)^2  \leq \frac{-w'(t_0 )}{\lambda (1-\delta )d}.$$

And then 

$$\mathcal{D} (w)(t_1) \geq \frac{-\lambda}{4} (1- w(t_1)^2  )^2 \geq \frac{-(w'(t_0 ))^2}{4\lambda (1-\delta )^2 d^2}$$

\medskip

In case $w' \geq 0$ on $I$,  we have that 

$$w'( t_1 ) \geq w'(t_1 ) - w'(t_0 ) = \int_{t_0}^{t_1} w''(t) dt = \int_{t_0}^{t_1} -H(t)w'(t) dt + \lambda  \int_{t_0}^{t_1} (-w(t)) (1-w(t)^2) dt$$ 

$$ \geq H(t_0 ) (w(t_0 ) -w(t_1 )) +\lambda (1-\delta )
 (1-w(t_0 )^2) d .$$

Therefore 

$$1-w(t_0 )^2 \leq \frac{1}{\lambda (1-\delta )d}  \left( w'(t_1 ) +H(t_0 )(w(t_1 )- w(t_0 )) \right) .$$

And then

$$\mathcal{D} (w) (t_0) \geq \frac{-1}{4 \lambda (1-\delta )^2 d^2}   \left( w'(t_1 ) +H(t_0 )(w(t_1 )- w(t_0 ) ) \right)^2 .$$

\end{proof}

\begin{lemma}\label{closeto1}
Let $\delta$, $\varepsilon$ be small positive numbers. Consider an interval $I=[t_0 , t_1 ] \subset [T-\varepsilon ,T]$ and let $d =t_1 - t_0$.
Assume that $w$ is a solution of Equation \ref{ODE3} such that 
$-1 < w(t) < -1 + \delta $ for all $t \in I$. Assume also that $w(t_z )=0$ for some $t_z <T-\varepsilon$ so that the corresponding function $\mathcal{D}(w)$
 is bounded away from
zero as in Lemma \ref{discrepancybound}. Then if $\lambda$ is large enough we have that $d \leq \frac{4}{\lambda^{1/5}}$.

\end{lemma} 

\begin{proof}

Assume $d > \frac{4}{ \lambda^{1/5}}$. We will reach a contradiction for $\lambda$ large enough. Note that $w$ can have at most 
one critical point in $I$, a local minimum. Therefore we can divide $I$ in at most two intervals such that $w$ is monotone on
each interval. We will reach a contradiction if the length of any of these intervals is greater than $\frac{2}{\lambda^{1/5}}$

Assume first that $w' \geq 0$ on the interval. 
Assume that the length  $d > \frac{2}{ \lambda^{1/5}}$. Let $t_2$ be the middle point of $I$.
Assume now that $w' >   \lambda^{1/5}$ on $[t_2 , t_1 ]$. Since $t_1 -t_2 > \frac{1}{ \lambda^{1/5}}$ this would imply that $w(t_1 ) >0$, which is not
true.  Then there is a a point $t_3 \in [t_2 , t_1 ]$  such that $w'(t_3 ) \leq   \lambda^{1/5}$. 
For $\lambda$ large enough we have that 
$\lambda^{1/5} \geq H(t_0 )(w(t_3 )- w(t_0 ) )$ and therefore we can assume that $\left( w'(t_3 ) +H(t_0 )(w(t_3 )- w(t_0 )) \right)^2 \leq  4  \lambda^{2/5}$.
Since $t_3 - t_0 >\frac{1}{\lambda^{1/5}}$
it follows from Lemma \ref{closeto1discrepancy}, applied to the interval $[t_0 , t_3 ]$,  that 

$$\mathcal{D}(w)(t_0 ) \geq -\frac{   \lambda^{4/5}}{  (1-\delta )^2 \lambda }.$$

For $\lambda$ large enough 

$$\frac{- 1  }{ (1-\delta )^2 \lambda^{1/5}  } > \frac{-H(T-\varepsilon)}{4(T-\varepsilon )} \geq  \frac{-H(t_z )}{4t_z}  .$$

Then it follows from Lemma \ref{discrepancybound} that $\mathcal{D}(w) (t_0 ) > \mathcal{D}(w) (t_z )$.  This is a contradiction since $\mathcal{D}$ is decreasing.

\medskip

Assume now  that $w' \leq 0$ and the length of the interval is  $d > \frac{2}{ \lambda^{1/5} }$. Let $t_2$ be the middle point of the
interval. If $w'(t) < - \lambda^{1/5}$ for all $t\in [t_0 , t_2 ]$ then $w(t_2 ) - w(t_0 ) <-1$ which is not true. Therefore there
exists $t_3 \in [t_0 , t_2 ]$ such that $w'(t_3 ) \geq - \lambda^{1/5}$. Applying Lemma 4.3 to the interval $[t_3 , t_1 ]$ we obtain

$$\mathcal{D} (w) (t_1 ) \geq \frac{ - (w'(t_3 )^2 )}{4\lambda (1-\delta )^2 (t_1 - t_3 )^2} \geq \frac{-1}{4 (1-\delta)^2 \lambda^{1/5}}.$$

\noindent
Applying Lemma \ref{closeto1discrepancy} we would see that for $\lambda$ large enough we have that $\mathcal{D}(w) (t_1 ) > \mathcal{D}(w) (t_z )$, which is a contradiction.

\end{proof}

\medskip

Next we have to bound the length of an interval where a solution $w$ of Equation \ref{ODE3} is bounded away from
$\pm 1$. For this we will apply Sturm comparison techniques (see for instance the book by L. Ince \cite{Ince}). In the following remark we summarize what we will use.

\begin{remark} Let $c\in (0,T)$.  
Consider a nonconstant solution $w$ of Equation \ref{ODE3}.  We will study the number of  zeroes of $w$
in an interval contained in $[c,T]$.
Let $r (x ) = e^{\int_{c}^{x} H (t) dt } >0$ and $p (x) = \lambda (1-w(x)^2 ) r(x)$.

Then $w$ is solution of the equation

\begin{equation}\label{3}
(r (x) w'(x))' +p (x) w(x) =0.
\end{equation}

Note that $r$ is increasing and $r(c)=1$. Let $r_0 = r(T)$.
 Note that $r(x) \in [1,  r_0 ]$ for any $x\in [c,T]$. 
Assume that $1-w(x)^2 \geq \delta$ so that $p(x) \geq \lambda \delta$.

Consider the equation

\begin{equation}\label{4}
(r_0  y'(x))' +\lambda \delta  y(x) =0.
\end{equation}

Note that $y$ is a solution if and only if $y'' +\frac{\lambda \delta}{r_0} y =0$. Let $\mu = \frac{\lambda \delta}{r_0}$.
All solutions of this 
equation are linear combinations of 
$\cos (\sqrt{\mu} x )$ and $\sin (\sqrt{\mu} x)$.

Now assume that the solution $w$ of Equation  \ref{ODE3}  verifies on an interval $I=[t_0 , t_1 ] \subset [c,T]$ that
$1-w(x)^2 \geq \delta$. Let $y$ be the solution of Equation \ref{4} such that $y(t_0 )= w(t_0 )$ and $y' (t_0 ) = w'(t_0 )$.
On any interval of length $\frac{2\pi}{\sqrt{\mu} }$ the solution $y$ must have at least 2 zeroes. 

Then it follows from the classical Sturm-Picone comparison theorem that if the length of $I$ is at least 
 $\frac{2\pi}{\sqrt{\mu} }$ then $w$ must have at least two zeroes in $I$. In particular, this cannot happen if $w$ is monotone.

\end{remark}

\begin{lemma}\label{F}
Let $c \in (0, T)$, $\delta >0$, $\varepsilon >0$ be fixed. Let $w$ be a solution of Equation \ref{ODE3} with at most $n_0\in N$ zeroes. 
Assume that $1-w(x)^2 \geq \delta$ on an interval $I$  of length $d$ contained
in $(c,T]$. If  $\lambda$ is large enough then $d<\varepsilon$.

\end{lemma}

\begin{proof} If $t_1$, $t_2$ are consecutive critical points of $w$, then $w$ is monotone on $[t_1 , t_2 ]$ and it has exactly one zero in $(t_1 , t_2 )$.
It follows in particular that there are at most $n_0$ such intervals contained in $I$. 
It follows from the previous remark that there is a positive constant $C$ such that $t_2 -t_1 <\frac{C}{\sqrt{\lambda}}$. Then it follows that
for $\lambda$ large enough
the sum of the lengths of all maximal intervals of monotonicity of $w$ contained in $I$ is less than $\varepsilon$, which implies that $d<\varepsilon$.

\end{proof}

We are now ready to prove Theorem \ref{ODE-interface}:

\begin{proof} (Theorem \ref{ODE-interface}) We can assume that $w(0) >0$. It follows from Lemma \ref{elementary} and Lemma \ref{F} that $w(0)$ is close to 1 if
$\lambda$ is large. 
Assume that $w(t)=0$ for some $t\in (0,T-\varepsilon ]$. Fix any positive, small $\delta$. It follows from Lemma \ref{closeto1} that
any interval contained in $(T-\varepsilon, T]$ where $1-w^2 <\delta$, has length at most $\frac{4}{\lambda^{1/5}}$, if $\lambda$ is large enough. 
On the other hand it follows from the last lemma that for $\lambda$ large enough the length of intervals
where  $1-w^2 \geq \delta$ is less than $\frac{\varepsilon}{2(n_0 +1)}$. But the interval $(T-\varepsilon , T]$ can be divided into at most $n_0 +1$ intervals
where $1-w^2 <\delta$ and at most $n_0 +1$ intervals where $1-w^2 \geq \delta$. From the previous comments it follows that
if $\lambda$ is large enough the sum of the legth of all such intervals is less than $\varepsilon$, which gives a contradiction.

\end{proof}

\section{Energy}

For a solution $w$ of Equation \ref{ODE2} we define the energy of $w$ as

$$\mathcal{E} (w)= \frac{1}{\sqrt{\lambda}}  \int_0^D \frac{w'(t)^2}{2}  dt + \frac{  \sqrt{\lambda} }{4} \int_0^D (1-w^2(t))^2 dt .$$

\bigskip

In this section we will prove that there is a uniform bound on the energy of solutions with a fixed number of zeroes:

\begin{theorem} Given $n_0 \in N$ there exists a constant $C>0$ such that for any solution $w$ of Equation \ref{ODE2} which has exactly 
$n_0$ zeroes, we have that
$\mathcal{E}(w)\leq C$.

\end{theorem}

\begin{proof} Let $w$ be a solution of Equation \ref{ODE2} with $n_0$ zeroes. Recall the discrepancy function $\mathcal{D}(w):[0,D] \rightarrow \re$:

$$\mathcal{D} (w) (t) = \frac{1}{2}  w'^2 (t) - \frac{\lambda }{4} (1- w^2 (t) )^2 .$$ 

We will first find a bound for the first summand of $\mathcal{E} (w)$:

\medskip

Recall that we had that  $\mathcal{D} (w)\leq 0$ on the whole interval $[0,D]$, and therefore  $\frac{1}{\sqrt{\lambda}}  \frac{w'(t)^2}{2} \leq  \frac{\sqrt{\lambda} }{4} (1-w^2(t) )^2 \leq \frac{\sqrt{\lambda}}{4}$. Therefore $|w'| \leq \frac{ \sqrt{\lambda} }{\sqrt{2} }$. If $t_1 < t_2$ are consecutive critical points of $w$ so that $w'$ does not change 
sign on $[t_1 , t_2 ]$ we have that

$$\int_{t_1}^{t_2} \frac{w'(t)^2  }{2\sqrt{\lambda} }dt \leq   \frac{1}{ 2\sqrt{\lambda} } \int_{t_1}^{t_2} |w'(t) | \frac{ \sqrt{\lambda} }{\sqrt{2} } dt
=\frac{1}{2\sqrt{2}} \left|  \int_{t_1}^{t_2} w'(t) dt \right| \leq \frac{1}{\sqrt{2}} .$$

Therefore

$$\frac{1}{\sqrt{\lambda}}  \int_0^D \frac{w'(t)^2}{2}  dt  \leq \frac{n_0}{\sqrt{2}}.$$

Note that the bound is valid for every $\lambda$.

\medskip

Now we use this bound for the first summand  to prove that the second summand of $\mathcal{E}(w)$ is also bounded. Fix any $c\in (0,a_0)$. We will actually use $c=\frac{a_0}{4}$. Then,

$$\frac{\mathcal{D}(w) (a_0)}{\sqrt{\lambda}}- \frac{\mathcal{D} (w) (c)}{\sqrt{\lambda}} = 
\frac{1}{\sqrt{\lambda}} \  \int_c^{a_0} \mathcal{D}(w)' (t) dt = \int_c^{a_0} \frac{-h(t)}{\sqrt{\lambda}} w'(t)^2 dt \geq -h(c) 
\sqrt{2}n_0 .$$

Since $\mathcal{D}(w)$ is decreasing in $(0,a_0)$,  for any $t\in (0,a_0]$ we have

$$\frac{\mathcal{D}(w) (t)}{\sqrt{\lambda}}      \geq \frac{\mathcal{D}(w) (a_0)}{\sqrt{\lambda}} \geq -h(c)\sqrt{2} n_0  +  \frac{\mathcal{D} (w) (c)}{\sqrt{\lambda}} .$$

Then,

$$ \frac{  \sqrt{\lambda} }{4} \int_0^D (1-w^2(t) )^2 dt = \frac{1}{2\sqrt{\lambda}} \int_0^D  w'(t)^2   dt - \frac{1}{\sqrt{\lambda}} \int_0^D  \mathcal{D}(w) (t) dt
 \leq    \frac{n_0}{\sqrt{2}}  + D \left( h(c) \sqrt{2}n_0 -  \frac{\mathcal{D} (w) (c)}{\sqrt{\lambda}} \right) .$$

For any fixed sufficiently small $\varepsilon >0$, it follows from Theorem \ref{ODE-interface} and Lemma \ref{F}
that there exists  a constant $\lambda_0 >0$ such  if $\lambda \geq \lambda_0$, then any solution
$w$ of Equation \ref{ODE2} with $n_0$ zeroes is decreasing on $(0, a_0-\varepsilon)$ and satisfies 
$w(a_0-\varepsilon )  \geq 1-\varepsilon$. 
Take for instance $\varepsilon = \frac{ \min \{ 1, a_0 \} }{10}$.
Obviously, for $\lambda \leq \lambda_0$ we have that

$$ \frac{  \sqrt{\lambda} }{4} \int_0^D (1-w^2(t))^2 dt \leq \frac{D \sqrt{\lambda_0}}{4}.$$

Now assume that $\lambda >\lambda_0$. We let $t_0 =\frac{a_0}{4}$ and note that by the choice of $\varepsilon$ there exists $t_1 \in  (\frac{a_0}{2}, a_0 -\varepsilon )$ such that $w'(t_1 ) > -1$.  Indeed, if $w'(t)\leq -1$ for every $t\in(a_0/2,a_0-\varepsilon)$, then $
w(a_0-\varepsilon)-w(a_0/2)
\leq
-\left(\frac{a_0}{2}-\varepsilon\right).$ On the other hand, since $1-\varepsilon\leq w<1$ on this interval,
$
w(a_0/2)-w(a_0-\varepsilon)\leq \varepsilon,
$
which contradicts the choice $\varepsilon\leq a_0/10$. 

Then it follows from Lemma \ref{closeto1discrepancy} applied to the interval $[t_0 , t_1 ]$ that
$$\frac{\mathcal{D} (w) (a_0 /4 ) }{\sqrt{\lambda}}  \geq \frac{-1}{4 \lambda^{3/2}  (1-\varepsilon )^2 (a_0 /4 )^2}   \left( 1 +h(a_0 /4 )) \right)^2 .$$ 

And therefore if $\lambda \geq \lambda_0$ we have that

$$ \frac{  \sqrt{\lambda} }{4} \int_0^D (1-w^2(t) )^2 dt \leq    \frac{n_0}{\sqrt{2}} + D  \left(  h(a_0 /4) \sqrt{2}n_0   + 
 \frac{ 8   \left( 1 +h(a_0 /4 )) \right)^2    }{ a_0^2 \lambda_0^{3/2} } \right) .$$

Therefore the theorem follows by taking

$$C=  \frac{n_0}{\sqrt{2}} +\min \left(  \frac{D \sqrt{\lambda_0}}{4}     ,   \frac{n_0}{\sqrt{2}} + D \left(   h(a_0 /4) \sqrt{2}n_0 + 
 \frac{ 8    \left( 1 +h(a_0 /4 )) \right)^2    }{ a_0^2 \lambda_0^{3/2} } \right)  \right)  .$$

\end{proof}

\end{document}